\documentclass[a4paper,12pt,reqno]{amsart}

\usepackage[all]{xy}           
\usepackage{amssymb}           
\usepackage{hyperref}
\usepackage{eucal}
\numberwithin{equation}{section}

\newtheorem{definition}{Definition}[section]
\newtheorem{lemma}[definition]{Lemma}
\newtheorem{theorem}[definition]{Theorem}
\newtheorem{proposition}[definition]{Proposition}
\newtheorem{corollary}[definition]{Corollary}
\newtheorem{remarkth}[definition]{Remark}
\newtheorem{ex}[definition]{Example}
\newenvironment{remark}
{\begin{remarkth}\upshape}{\hfill$\diamond$\end{remarkth}}
\newenvironment{example}
{\begin{ex}\upshape }{\hfill $\triangleleft$\end{ex}}

\renewcommand{\emph}[1]{{\bfseries\itshape{#1}}}

\newcommand{\R}{\mathbb{R}}      

\newcommand{\ds}{\displaystyle}

\begin{document}

\title[Lagrangian submanifolds in $k$-symplectic settings]{Lagrangian submanifolds in $k$-symplectic settings}

\author[M. de Le\'on]{M. de Le\'on}
\address{M. de Le\'on:
Instituto de Ciencias Matem\'aticas (CSIC-UAM-UC3M-UCM),
C/ Nicol\'as Cabrera 13,15, 28049 Madrid, Spain} \email{mdeleon@icmat.es}

\author[S. Vilari\~no]{S. Vilari\~no}
\address{S. Vilari\~no:
Centro Universitario de la Defensa-IUMA, Academia General Militar, Carretera de Huesca, s/n, 50090-Zaragoza, Spain} \email{silviavf@unizar.es}

\today{ }



\begin{abstract}
    In this paper we extend the well-know normal form theorem for Lagrangian submanifolds proved by A. Weinstein in symplectic geometry to the setting of $k$-symplectic manifolds.

\end{abstract}

\thanks{This work has been partially supported by MICINN (Spain) MTM2010- 21186-C02-01, MTM2010-12116-E and MTM2011-2585, the European project IRSES-project ``Geomech-246981'' and the ICMAT Severo Ochoa project SEV-2011-0087.}

 \maketitle

\tableofcontents

\section{Introduction}

    In symplectic geometry, the classical Darboux theorem states that if $M$ is a symplectic manifold of dimension $2n$ with a symplectic form $\omega$ and $x$ is any point of $M$, then there exists a coordinate system $(x^i, y^i)$ on a neighborhood $U$ of $x$ such that $\omega=dx^i\wedge dy^i$ on $U$.

    In \cite{W_1971} A. Weinstein has generalized this result in the following sense:
        \begin{theorem}
            Let $N$ be a closed submanifold of $M$ and $\Omega_0$, $\Omega_1$  symplectic structures on $M$ such that $\Omega_0\vert_N=\Omega_1\vert_N$. Then there exists a automorphism $f\vert_N$ on $M$ such that $f=id_N$ and $f^*\Omega_1=\Omega_0$.
        \end{theorem}

    The above theorem is very useful in order to understand the geometrical properties of lagrangian submanifolds. Indeed,
    A. Weinstein has obtained the following result, which gives the normal form for a lagrangian submanifold  in a symplectic manifold.
        \begin{theorem}
            Let $(\mathcal{P},\omega)$ be a symplectic manifold and let $\mathcal{L}$ be a lagrangian submanifold. Then there exists a tubular neighborhood $U$ of $\mathcal{L}$ in $\mathcal{P}$, and a diffeomorphism $\phi\colon U\to V=\phi(U)\subset T^*\mathcal{L}$ into an open neighborhood $V$ of the zero cross-section in $T^*\mathcal{L}$ such that $\phi^*(\omega_\mathcal{L}\vert_{V})=\omega\vert_{U}$, where $\omega_\mathcal{L}$ is the canonical symplectic form on $T^*\mathcal{L}$.
        \end{theorem}

    A natural generalization of a symplectic manifold is the notion of the so-called $k$-symplectic manifolds. Recall that a $k$-symplectic structure on a manifold $M$ of dimension $n(k+1)$ is a family of $k$ closed $2$-forms $(\omega_1,\ldots, \omega_k)$ on $M$ such that $\cap_{r=1}^k\ker\omega_r=\{0\}$. Let us observe that the canonical model of a symplectic manifold is the cotangent bundle $T^*Q$, while the canonical model of $k$-symplectic manifold is the bundle of $k^1$-covelocities, that is the Whitney sum $T^*Q\oplus \stackrel{k}{\cdots} \oplus T^*Q$ of $k$ copies of the contangent bundle. The notion of $k$-symplectic structure was independently introduced  by A. Awane \cite{Aw-1992, Aw-1998}, G. G\"{u}nther \cite{Gu-1987},  M. de Le\'{o}n \textit{et al.} \cite{LMS-1988, LMS-1993}, and L.K. Norris \cite{MN-2000, Norris-1993}. Let us recall that $k$-symplectic manifolds provide a natural arena to develop classical field theory as an alternative to other geometrical settings, like multisymplectic manifolds (see \cite{arxive_mdl, RRSV_2011} for a survey on this subject).

    A Darboux theorem for $k$-symplectic manifolds has been proved in \cite{Aw-1992,LMS-1993}, but in this case, for $k\geq 2$ one needs an additional geometric ingredient. Indeed, it is necessary to assume the existence of a $k$-lagrangian integrable distribution of dimension $nk$ (let us observe that we can extend the usual concepts of isotropic, coisotropic and lagrangian submanifolds  of the symplectic geometry to the case of $k$-symplectic manifolds, see section \ref{section_lagrangian}). This fact leads us to introduce the notion of {\it polarized} $k$-symplectic manifold as a $k$-symplectic manifold $(M,\omega_1,\ldots, \omega_k)$ such that there exist $W$ a $k$-lagrangian integrable distribution of codimension $n$.

    The natural question is if we can generalize the Weinstein's normal form theorem for lagrangian submanifolds to this new geometric setting.

    The main result of the paper is stated as follows (see section \ref{normal_form_section}):

        \begin{theorem}
        Let $(M,\omega_1,\ldots, \omega_k, W)$ be a polarized $k$-symplectic manifold. Let $\mathcal{L}$ be a $k$-lagrangian submanifold which is complementary to $W$, that is,  $T\mathcal{L}\oplus W\vert_\mathcal{L}=TM\vert_\mathcal{L}$. Then there is a tubular neighborhood $U$ of $\mathcal{L}$ and a diffeomorphism $\Phi\colon U\to V\subset T^*\mathcal{L}\oplus\stackrel{k}{\cdots}\oplus T^*\mathcal{L}$ where $V$ is a neighborhood of the zero section, such that $\Phi\vert_\mathcal{L}$ is the standard identification of $\mathcal{L}$ with the zero section of $T^*\mathcal{L}\oplus\stackrel{k}{\cdots}\oplus T^*\mathcal{L}$, and
        \[
            \Phi^*\left(\Omega^\mathcal{L}_r\vert_V\right)=
            \omega_r\vert_U,
        \]for all $r\in\{1,\ldots, k\}$, where $(\Omega^\mathcal{L}_1,\ldots, \Omega^\mathcal{L}_k)$ is the canonical $k$-symplectic structure on $T^*\mathcal{L}\oplus\stackrel{k}{\cdots}\oplus T^*\mathcal{L}$.
    \end{theorem}

Notice that the above result can be compared to the corresponding one in the multsymplectic setting (see \cite{LMS-2007, Martin}).

    As a first consequence of the above results, we can induce a differentiable structure  to the group  of a $k$-symplectomorphism of a (polarized) $k$-symplectic manifold on itself. Let us recall that given two $k$-symplectic manifolds  $(M_i,\omega^i_1,\ldots, \omega^i_k), \, i=1,2$, a diffemorphism $\phi$ from $M_1$ to $M_2$ is called a $k$-symplectomorphism if and only if $\phi^*\omega^2_r=\omega^1_r,\, 1\leq r\leq k$. Along this paper, given a $k$-symplectic manifold $(M,\omega_1,\ldots, \omega_k)$ we denote by $G(M,\omega_1,\ldots, \omega_k)$ the group of automorphism of $M$ which are $k$-symplectomorphism; this group, $G(M,\omega_1,\ldots, \linebreak\omega_k)$,  is called the automorphism group of the $k$-symplectic manifold $M$.

    Using that there is a one-to-one correspondence between the group of automorphism of a symplectic manifold which are close to the identity  and the family of $k$ closed $1$-forms on $M$, one can introduce a chart in the group of the diffeomorphism of $M$.

    We hope to use these results to go further in the study of the group of the automorphism of a $k$-symplectic manifold.

    Throughout this paper we shall assume that our vector space and manifolds are finite dimensional.

\section{$k$-symplectic vector spaces and special subspaces}

    By an exterior form (or simply a form) on a vector space, we mean an alternating multilinear function on that space with values in the field of scalars. The contraction of a vector $v\in V$ and an exterior form $\omega$ on $V$ will be denoted by $\imath_v\omega$.

    In this section we study the \(k\)-symplectic vector spaces. After some introductory definitions, a model of \(k\)-symplectic vector space is described. Finally we discuss the notion of orthogonal complement of a subspace of a \(k\)-symplectic vector space and study some subspace with special properties.

    \begin{definition}
        A \(k\)-symplectic vector space \((\mathcal{V},\omega_1,\ldots,\omega_k)\) is a vector space \(\mathcal{V}\) of dimension \(n(k+1)\) and a family of \(k\) skew-symmetric bilinear forms \(\omega_1,\ldots,\omega_k\) such that
        \begin{equation}\label{nondeg_cond}
            \bigcap_{r=1}^k\ker\,\omega_r=\{0\}\,,
        \end{equation}
        where \(\ker\,\omega_r=\{u\in \mathcal{V} |\, \omega_r(u,v)=0,\,\forall v\in \mathcal{V} \}\) denotes the kernel of \(\omega_r\).
    \end{definition}

    Note that for \(k=1\) the above definition reduces to that of a symplectic vector space. The condition (\ref{nondeg_cond}) means that the induced linear map
    \[
        \begin{array}{rccl}
            {\sf b}\colon & \mathcal{V} & \to & \mathcal{V}^*\times \stackrel{k}{\cdots}\times \mathcal{V}^*\\\noalign{\medskip}
             & v &\mapsto & (\imath_v\omega_1,\ldots, \imath_v\omega_k)
        \end{array}
    \]
    is injective or equivalently that it has maximal rank, that is,  $rank\, {\sf b}=\dim\,\mathcal{V}=n(k+1)$.


    \begin{example}\label{euclidean space}

        We consider the vector space $\mathcal{V}=\mathbb{R}^3$ with the the family of skew-symmetric bilinear forms
                \[
                    \omega_1=e^1\wedge e^3 \quad \makebox{and} \quad \omega_2=e^2\wedge e^3\,,
                \]
        where $\{e_1,e_2,e_3\}$ is the canonical basis of $\mathbb{R}^3$ and $\{e^1,e^2,e^3\}$ the dual basis. It is easy to check that
        \[
            \ker\,\omega_1=span\{e_2\} \makebox{ and } \ker \,\omega_2=span\{e_1\}
        \] and therefore
        \(
            \ker\,\omega_1\cap\ker\,\omega_2=\{0\}
        \), that is, $(\omega_1,\omega_2)$ is a $2$-symplectic structure on $\mathbb{R}^3$.

    \end{example}

    \begin{example}\label{r6}
        We consider the vector space $\mathcal{V}=\mathbb{R}^6$ with the family of skew-symmetric bilinear forms
        \[
            \omega_1= e^1\wedge e^3 + e^4\wedge e^6 \makebox{ and } \omega_2= e^2\wedge e^3 + e^5\wedge e^6
        \]
        where $\{e_1,e_2,e_3,e_4,e_5,e_6\}$ is the canonical basis of $\mathbb{R}^6$ and $\{e^1,\ldots,e^6\}$ the dual basis. It is easy to check that
        \[
            \ker\,\omega_1=span\{e_2, e_5\} \makebox{ and } \ker \,\omega_2=span\{e_1, e_4\}
        \] and therefore
        \(
            \ker\,\omega_1\cap\ker\,\omega_2=\{0\}
        \), that is, $(\omega_1,\omega_2)$ is a $2$-symplectic structure on $\mathbb{R}^6$.

        Another $k$-symplectic structure on $\mathbb{R}^6$ is given by the family of $2$-forms $\omega^r=e^r\wedge e^6,\, r\in \{1,\ldots, 5\}$, which is a $5$-symplectic structure on $\mathbb{R}^6$.
    \end{example}
    \begin{example}\label{canonical_model}
        It is well-known that for any vector space $V$, the space $V\times V^*$ admits a canonical symplectic form $\omega_V$ given by
        \[
            \omega_V\left( (v,\alpha), (w,\beta)\right)= \beta(v)-\alpha(w)\,,
        \]
        for $v, w\in V$ and $\alpha,\beta\in V^*$ (see for instance \cite{AM-1978}). This structure has the following natural extension to the $k$-symplectic setting. For any $k$, the space $V\times V^*\times\stackrel{k}{\cdots}\times V^*$ can be equipped with a family of $k$ canonical skew-symmetric bilinear forms $(\omega^V_1,\ldots, \omega^V_k)$ given by
        \begin{equation}\label{canonical_omega}
            \omega^V_r\left( (v,\alpha_1,\ldots, \alpha_k),(w,\beta_1,\ldots, \beta_k)\right)=\beta_r(v)-\alpha_r(w)\,,
        \end{equation}
        for $v,w\in V$ and $(\alpha_1,\ldots, \alpha_k), (\beta_1,\ldots,\beta_k)\in V^*\times\stackrel{k}{\cdots}\times V^*$. It is now a simple computation to show that $(V\times V^*\times\stackrel{k}{\cdots}\times V^*, \omega^V_1,\ldots,\omega^V_k)$ is a $k$-symplectic vector space.

        Let us observe that if we consider the natural projection
            \[
                pr_r\colon (v,\alpha_1,\ldots,\alpha_k)\in V\times V^*\times \stackrel{k}{\cdots}\times V^*\to (v,\alpha_r)\in V\times V^*,
            \]
            the $2$-form $\omega^V_r$ is exactly $pr_r^*\omega_V$.
    \end{example}
    \begin{definition}
        Let $(\mathcal{V}_1,\omega^1_1,\ldots, \omega^1_k)$ and $(\mathcal{V}_2,\omega^2_1,\ldots, \omega^2_k)$ be two $k$-symplectic vector space and let $\phi\colon \mathcal{V}_1\to \mathcal{V}_2$ be a linear isomorphism. $\phi$ is called a $k$-symplecto\-mor\-phism if it preserves the $k$-symplectic structure, that is, $\phi^*\omega^2_r=\omega^1_r,\;\forall r\in\{ 1,\ldots, k\}$.
    \end{definition}


    On a symplectic vector space there is a natural notion of orthogonal complement of a subspace with respect to the given symplectic $2$-form. On a $k$-symplectic space, however, there are several options for defining some kind of ``orthogonality''. Indeed, let $(\mathcal{V},\omega_1,\ldots, \omega_k)$ be an arbitrary $k$-symplectic vector space, and let $W$ be a subspace of $\mathcal{V}$. For each $l$, with $1\leq l\leq k$, the \textit{$l$-th orthogonal complement} of $W$ is the linear subspace of $\mathcal{V}$ defined by
    \begin{equation}\label{l-orthogonal}
        W^{\bot,l}=\{ v\in \mathcal{V}\, |\, \omega_1(v,w)=\ldots =\omega_l(v,w)=0,\, \makebox{for all } w\in W\}\,.
    \end{equation}

    The following proposition collects some properties related to the above definition.

    \begin{proposition}\label{properties orthogonal}\
        \begin{enumerate}
            \item $W^{\bot,k}\subseteq W^{\bot,k-1}\subseteq \ldots \subseteq W^{\bot,1}$.

            \item For all index $l\in \{1,\ldots, k\}$ the following properties hold.
                \begin{enumerate}
                    \item $\{0\}^{\bot,l} = \mathcal{V}$.
                    \item If $V\subset W$ then $ W^{\bot, l}\subset V^{\bot, l}$.
                    \item $W\subset (W^{\bot, l})^{\bot, l}$.
                    \item $W\cap W^{\bot,l}=\displaystyle\bigcap_{r=1}^l\ker \,(\omega_r\vert_W)\,$. As a consequence one has
                             \[
                                \mathcal{V}^{\bot,l}=
                                \displaystyle\bigcap_{r=1}^l\ker\,
                                \omega_r \makebox{ and } \mathcal{V}^{\bot,k}=\{0\}.
                             \]
                    \item $ (V+W)^{\bot, l} \subset V^{\bot, l}\cap W^{\bot, l}$.
                \end{enumerate}

            \item Finally for all $l_1,l_2\in \{1,\ldots, k\}$,
                \begin{enumerate}
                    \item $V^{\bot, l_1} + W^{\bot, l_2} \subset (V\cap W)^{ \bot, \min\{l_1, l_2\}}$.
                    \item $V^{\bot, l_1} \cap W^{\bot, l_2}\subset (V+W)^{\bot, \min\{l_1,l_2\}}$.
                \end{enumerate}
        \end{enumerate}
    \end{proposition}
    \begin{proof}
        The assertions (i), (ii.a) and (ii.b) are simple verifications. For (ii.c)  by definition of $W^{\bot, l}$, for all $v\in W^{\bot, l}$, one obtains that if $w\in W$, then $\omega_1(v,w)=\cdots = \omega_l(v,w)=0$, i.e. $w\in (W^{\bot, l})^{\bot, l}$.

        To prove (ii.d) we consider that
        \begin{align*}
                 & W\cap W^{\bot, l} \\
                 = & W\cap \{v\in \mathcal{V}\, | \, \omega_1(v,w)=\cdots =\omega_l(v,w)=0,\makebox{ for all } w\in W\} \\
                  = &\{v\in W\, | \, \omega_1(v,w)=\cdots = \omega_l(v,w)=0,\makebox{ for all } w\in W\}
                 \\
                  = &\displaystyle\bigcap_{r=1}^l \{v\in W\, | \, \omega_r(v,w)=0,\makebox{ for all } w\in W\}
                 = \displaystyle\bigcap_{r=1}^l\ker\, (\omega_r\vert_W)\,.
            \end{align*}

        In particular, one has that
        \[
            \mathcal{V}^{\bot,l} = \mathcal{V}\cap \mathcal{V}^{\bot,l} = \displaystyle\bigcap_{r=1}^l\ker\,\omega_r\,.
        \]

        Finally, when $l=k$, from the above expression and the $k$-symplectic condition (\ref{nondeg_cond}) one obtains
        \(
            \mathcal{V}^{\bot, k}=\displaystyle\bigcap_{r=1}^k\ker\,\omega_r=\{0\}\,.
        \)

        For (ii.e), since  $V\subset V+W$ and $W\subset V+W$, from (ii.b) one obtains that $(V+W)^{\bot, l}\subset V^{\bot, l}$ and $(V+W)^{\bot, l}\subset W^{\bot, l}$, so $(V+W)^{\bot, l}\subset V^{\bot, l}\cap W^{\bot, l} $.

         \smallskip

         To prove (iii) we consider two index $l_1,l_2\in\{1,\ldots, k\}$. Since $V\cap W\subset V$ and $V\cap W\subset W$, as a consequence of (i) and (ii.b) one has $$ V^{\bot, l_1}\subset (V\cap W)^{\bot, l_1}\subset (V\cap W)^{\bot, \min\{l_1,l_2\}} $$ and $$  W^{\bot, l_2}\subset (V\cap W)^{\bot, l_2}\subset (V\cap W)^{\bot, \min\{l_1,l_2\}}.$$ Therefore  $V^{\bot, l_1} + W^{\bot, l_2} \subset (V\cap W)^{ \bot, \min\{l_1, l_2\}}$.

         Now, for (iii.b), let $u\in V^{\bot, l_1}\cap W^{\bot, l_2}$ be, then $ \omega_1(u,v) = \cdots = \omega_{l_1}(u,v)=0,$ for all $v\in V$ and $ \omega_1(u,w) = \cdots = \omega_{l_2}(u,w) =0$ for all $w\in W$. Thus, we have that
                \(
                    \omega_1(u,\lambda v+ \mu w)=\cdots = \omega_{\min\{l_1,l_2\}}(u,\lambda v + \mu w) =0,\) for all $\lambda v+ \mu w\in V+W $, that is
                $u\in (V+W)^{\bot, \min\{l_1,l_2\}}$, and therefore $V^{\bot, l_1}\cap W^{\bot, l_2}\subset (V+W)^{\bot, \min\{l_1,l_2\}}$.
    \end{proof}
    \begin{remark}
        In the symplectic case ($k=1$), the above relations reduce to the well-know properties of the orthogonal complement of a subspace \cite{AM-1978}.
    \end{remark}
    \begin{remark}
        In general $W\neq (W^{\bot, l})^{\bot, l}$. In fact, we consider the vector space $\mathcal{V}=\mathbb{R}^3$ with the $2$-symplectic structure given by
                \[
                    \omega_1=e^1\wedge e^3 \quad \makebox{and} \quad \omega_2=e^2\wedge e^3\,,
                \]
        where $\{e_1,e_2,e_3\}$ is the canonical basis of $\mathbb{R}^3$ and $\{e^1,e^2,e^3\}$ the dual basis. If $W=\{0\}$, then
                \[
                    (\{0\}^{\bot, 1})^{\bot, 1}=(\mathbb{R}^3)^{\bot,1}=\ker\,\omega_1=span\{e_2\}\neq \{0\}\,.
                \]
    \end{remark}

    Now, putting $l=l_1=l_2$ in (iii.b) of the previous Proposition, we derive the following result.

    \begin{corollary}\label{corol1}
     For any two subspaces $V$ and $W$ of a $k$-symplectic vector space $(\mathcal{V},\omega_1,\ldots, \omega_k)$, we have
        \begin{enumerate}
            \item $V^{\bot, l}\cap W^{\bot, l}=(V+W)^{\bot, l}$.
            \item $\left( (V^{\bot, l}+W^{\bot, l})^{\bot, l} \right)^{\bot, l} \subset  (V\cap W)^{\bot, l}.$
         \end{enumerate}
    \end{corollary}
    \begin{proof}
    From (ii.e) and (iii.b) of the Proposition \ref{properties orthogonal} one obtains
    \[
        (V+W)^{\bot, l}\subset V^{\bot, l}\cap W^{\bot, l}\subset (V+W)^{\bot, l}\,,
    \]
    then $(V+W)^{\bot, l}= V^{\bot, l}\cap W^{\bot, l}$.

    For (ii), notice that using  (ii.c) of the previous Proposition  and (i) here,
         \[
            V\cap W \subset (V^{\bot, l})^{\bot, l} \cap (W^{\bot, l})^{\bot, l} = (V^{\bot, l} + W^{\bot, l})^{\bot, l}.
         \]

         So, from (ii.b), (ii) holds.
    \end{proof}

    \begin{remark}

        In the $k$-symplectic setting there is an important difference with the symplectic case. In this setting is not true that
        \[
            \dim W + \dim W^{\bot, k} = \dim \mathcal{V}.
        \]
        In fact, we consider the $2$-symplectic vector space of the example \ref{euclidean space} and the subspace $W=span\{e_3\}$. It is trivial that $W^{\bot, 2}=span\{e_3\}^{\bot, 2}=span\{e_3\}$, then
        \[
            \dim W + \dim W^{\bot, 2}=2\dim span\{e_3\}= 2 \neq 3=\dim \mathbb{R}^3\,.
        \]

        Another example is the following. Consider the $k$-symplectic vector space $\mathcal{V}=V\times V^*\times\stackrel{k}{\cdots}\times V^*$ and the subspace $W=V^*\times\stackrel{k}{\cdots}\times V^*$. It is easy to check that $W^{\bot, k}= W$, and then we obtain
        \[
            \dim W+ \dim W^{\bot, k}= 2\dim V^*\times \stackrel{k}{\cdots}\times V^*=2k\dim V
        \]
        and $\dim \mathcal{V}= (k+1)\dim V$, so  $\dim W + \dim W^{\bot, k} = \dim \mathcal{V}$ if and only if $k=1$.

        In general, we only have the trivial formula
        \[
            \dim W + \dim W^{\bot, l}\leq 2\dim \mathcal{V}\,.
        \]
    \end{remark}
\begin{subsection}{Isotropic, coisotropic and Lagrangian subspace}\label{section_lagrangian}
    We can now introduce the following special types of subspaces of a $k$-symplectic vector space, generalizing the corresponding notions in symplectic geometry, \cite{AM-1978}.
    \begin{definition}\label{subspaces}
        Let $(\mathcal{V},\omega_1,\ldots, \omega_k)$ be a $k$-symplectic vector space. A subspace $W$ of $\mathcal{V}$ is said to be
        \begin{enumerate}
            \item $l$-isotropic if $W\subset W^{\bot, l}$;
            \item $l$-coisotropic if $W^{\bot, l}\subset W$;
            \item $l$-lagrangian if $W$ is $l$-isotropic and there exists a $l$-isotropic subspace $U$ of $\mathcal{V}$, such that $\mathcal{V}=U\oplus W$.
        \end{enumerate}
    \end{definition}
    \begin{example}\label{euclidean example}
        We consider the vector space $\mathcal{V}=\mathbb{R}^3$ with the $2$-symplectic structure introduced in the example \ref{euclidean space}.
            \begin{itemize}
                \item $span\{e_2\}^{\bot, 1}=\mathbb{R}^3$ and $span\{e_2\}^{\bot, 2}=span\{e_1,e_2\}$, so $span\{e_2\}$ is a $1$-isotropic and $2$-isotropic subspace.
                \item $span\{e_1,e_3\}^{\bot, 2}=\{0\}$, then $span\{e_1,e_3\}$ is a $2$-coisotropic subspace.
                \item $span\{e_1\}$ is a $1$-lagrangian subspace. In fact, $span\{e_1\}\subset span\{e_1\}^{\bot, 1}= span\{e_1,e_2\}$ and there exists $U=span\{e_2,e_3\}$ such that $\mathbb{R}^3=span\{e_1\}\oplus  span\{e_2,e_3\}$ and $span\{e_2,e_3\}\subset span\{e_2,e_3\}^{\bot, 1}=span\{e_2,e_3\}$.
                \item $span\{e_3\}=span\{e_3\}^{\bot, 2},$ $\, span\{e_1,e_2\}^{\bot, 2}=span\{e_1,e_2\}$ and $\mathbb{R}^3=span\{e_1,e_2\}\oplus span\{e_3\}$. Therefore, $span\{e_3\}$ is a $2$-lagrangian subspace.
            \end{itemize}
    \end{example}
    \begin{example} A more general situation is the following. We consider the canonical model of $k$-symplectic vector space $\mathcal{V}= V\times V^*\times\stackrel{k}{\cdots}\times V^*$, described in the example \ref{canonical_model}, where $V$ is any vector space. Identifying $V$ and $V^{k*}=V^*\times\stackrel{k}{\cdots}\times V^*$ with the subspaces $V\times \{0\}$ and $\{0\}\times V^{k*}$, respectively, we have,
        \begin{enumerate}
            \item $V^{k*}=V^*\times\stackrel{k}{\cdots}\times V^*$ is a $l$-isotropic and $l$-coisotropic subspace for all $l\in \{1,\ldots, k\}$. In fact, it is straightforward to check that $(V^{k*})^{\bot, l}=V^{k*}$.
            \item The subspace $V\subset \mathcal{V}=V\times V^*\times\stackrel{k}{\cdots}\times V^*$ satisfies that $V^{\bot, k}=V$, then $V$ is a $k$-isotropic an a $k$-coisotropic subspace.
            \item $V$ and $V^*\times\stackrel{k}{\cdots}\times V^*$ are $k$-lagrangian subspaces.
        \end{enumerate}
    \end{example}

    The last two examples of $k$-lagrangian subspaces illustrate that  in a $k$-symplectic space, the $k$-lagrangian subspaces need not all have the same dimension. The situation is thus quite different from the finite dimensional symplectic case ($k=1$), where all lagrangian subspaces have the same dimension, namely half the dimension of the given symplectic space.

    In view of (i) of the Proposition \ref{properties orthogonal} it is clear that an $l$-isotropic (resp. $l$-coisotropic)  subspace is also $l'$-isotropic (resp. $l''$-coisotropic) for all $1\leq l'< l$ (resp. $l < l''\leq k$).

    In the next Proposition we collect a few interesting properties concerning the concepts of isotropic, coisotropic and lagrangian subspaces.

    \begin{proposition}\label{charac_subspace}
        Let $(\mathcal{V},\omega_1,\ldots, \omega_k)$ be a $k$-symplectic vector space and $W$ a subspace of $\mathcal{V}$.
        \begin{enumerate}
            \item Each subspace of dimension $1$ is $l$-isotropic for all $l\in \{1,\ldots, k\}$.
            \item Each subspace of codimension $1$ is $k$-coisotropic.
            \item $W$ is $l$-isotropic if, and only if,  $\omega_r\vert_{W\times W}=0$ for all $r\in\{1,\ldots, l\}$.
            \item If $W=W^{\bot, l}$, then $W$ is a $l$-lagrangian subspace, for each $l\in\{1,\ldots,r\}$.
            \item If $W$ is $k$-lagrangian then $W=W^{\bot, k}$.
            \item If $U$ is a $l$-isotropic subspace of $\mathcal{V}$, then for every $l'\leq l$ there exist a $l'$-lagrangian subspace which containts to $U$.
        \end{enumerate}
    \end{proposition}
    \begin{proof}

        (i)\; If $W$ is a $1$-dimensional subspace of $\mathcal{V}$ it readily follows from the definition of the $k$-orthogonal complement that $W\subset W^{\bot, k}$, i.e. $W$ is $k$-isotropic and then $l$-isotropic for every $l\leq k$.

        (ii)\; Let $W$ be a $(n(k+1)-1)$-dimensional subspace of $\mathcal{V}$ and let $\{e_1,\ldots, e_{n(k+1)-1}\}$ denote a basis of $W$. For any $v\in W^{\bot, k}$ we have, by definition
        \[
            \omega^r(v,e_i)=0,
        \]
        for all $1\leq i\leq n(k+1)-1$, $1\leq r\leq k$. If $0\neq v\not\in  W$, then it really follows that $\imath_v\omega_r=0,\, 1\leq r\leq k$, that is, $v\in \cap_{r=1}^k\ker\omega_r=\{0\}$, which it is a contradiction, therefore $W^{ \bot, k}\subset W$, i. e. $W$ is a $k$-coisotropic subspace.
        
        (iii)\; Let $W$ be a $l$-isotropic space, that is, $$W \subset W^{\bot, l}=\{ v\in \mathcal{V} \,|\, \omega_1(v,w)=\ldots =\omega_l(v,w)=0,\, \makebox{for all } w\in W\},$$ then for all $v, w \in W$ one has $\omega_1(v,w)=\ldots =\omega_l(v,w)=0$,  or equivalently $\omega_r\vert_{W\times W}=0\,,\; r\in \{1,\ldots, l\}$. Conversely, if $v\in W$, then $\omega_1(v,w)=\cdots=\omega_l(v,w)=0$, for all $w\in W$, then, by the definition \ref{l-orthogonal}, one has that $v\in W^{\bot,l}$, that is, $W\subset W^{\bot, l}$.

         (iv)\; We now suppose that $W=W^{\bot, l}$, we construct $U$ of definition \ref{subspaces}(iii) as follows. Choose arbitrarily $u_1\not\in W$ and let $U_1=span(u_1)$; since $W\cap U_1=\{0\}$, one has
         \[
            \mathcal{V}=\{0\}^{\bot, l}=(W\cap U_1)^{\bot, l} = W^{\bot, l} + U_1^{\bot, l}=W + U_1^{\bot, l}\,.
         \]
         Moreover, as $U_1$ has dimension $1$, it is $l$-isotropic.

         Now take a vector $u_2\in U_1^{\bot, l}, \, u_2\not\in W+U_1$, let $U_2=U_1+span\{u_2\}$. By construction $W\cap U_2=\{0\}, \, \mathcal{V}=W+U_2^{\bot, l}$ and from (i) in Corollary \ref{corol1}, one has
         \[
            U_2=U_1+span\{u_2\}\subset U_1^{\bot,l}\cap span\{u_2\}^{\bot, l}=(U_1 + span\{u_2\})^{\bot, l}=U_2^{\bot, l},
         \]since $u_2\in U_1^{\bot, l}$. If we continue inductively, we may construct a chain of $l$-isotropic subspaces $U_1\subset U_2\subset\cdots$, such that $W\cap U_i=\{0\}$ and $\mathcal{V}=W + U_i^{\bot, l}$. This chain necessarily possesses a maximal element $U_s$  for which $U_s=U_s^{\bot, l}$. Thus we can choose $U=U_s$.

        (v)\; We now prove that the assertion \textit{ $W$ is $k$-lagrangian} implies that $W=W^{\bot, k}$. We have $W\subset W^{\bot, k}$ by definition of $k$-lagrangian subspace. Conversely, let $v\in W^{\bot, k}$ and write $v=u+w$, where $u\in U$ and $w\in W$, being $U$ the subspace given in definition \ref{subspaces}(iii). We shall show that $u=0$. Indeed, since $U$ is, in particular, $k$-isotropic, $u\in U^{\bot, k}$. Similarly $u=v-w\in W^{\bot, k}$, then $u\in U^{\bot, k}\cap W^{\bot, k}= (U+W)^{\bot, k}=\mathcal{V}^{\bot, k}=0$. Thus, $u=0$, so $W^{\bot, k}\subset W$ and the identity $W^{\bot, k}=W$ holds.

        For (vi), it suffices to prove that every $l$-isotropic subspace $U$ is contained in a $l$-lagrangian subspace of $\mathcal{V}$ since, as noticed above, a $l$-isotropic subspace is also $l'$-isotropic for every $l'\leq l$.

        By assumption, $U\subset U^{\bot, l}$. If $U\neq U^{\bot, l}$, take a vector $v_1\in U^{\bot, l}$ such that $v_1\not \in U$. Let $U_1=U+span\{v_1\}$. By construction $U_1 \subset U_1^{\bot, l}$, indeed since $v_1\in U_1^{\bot, l}$,
        \[
            U_1=U\oplus span (v_1)\subset U^{\bot, l}\cap span\{v_1\}^{\bot, l} = (U + span\{v_1\})^{\bot, l}= U_1^{\bot,l}.
        \]
        Summarizing, we thus have the inclusions $U\subset U_1\subset U_1^{\bot, l}\subset U^{\bot, l}$. Continuing inductively, we may construct a chain of $l$-isotropic subspaces $U\subset U_1\subset U_2\subset\cdots$ which necessarily possesses a maximal element $W$ for which $W=W^{\bot, l}$. By (iv) we know that this condition implies that $W$ is $l$-lagrangian subspace.
    \end{proof}
    \begin{remark}
    Form items (iv) and (v) of the above proposition we obtain that a subspace $W$ of a $k$-symplectic vector space $(\mathcal{V},\omega_1,\ldots,\omega_k)$ is $k$-lagrangian if and only if $W=W^{\bot, k}$. Observe that this equivalence is only valid for the index $k$ (see for instance example \ref{euclidean example}). As a consequence of this characterization is easy to prove that a $k$-lagrangian subspace $U$ is maximal in the sense that there is not a another $k$-lagrangian subspace $V$ such that $U\subset V$. In fact, in that case, from item (ii.b) of Proposition \ref{properties orthogonal} one has,
    \[
        U\subset V\subset V^{\bot, k}\subset U^{\bot, k}=U
    \]
    and therefore $U=V$.
    \end{remark}
    \begin{proposition}\label{canonical_relation}
        Let $(\mathcal{V}_1,\omega^1_1,\ldots, \omega^1_k)$ and $(\mathcal{V}_2,\omega^2_1,\ldots, \omega^2_k)$ be two $k$-symplectic vector space and $\pi_i\colon \mathcal{V}_1\times \mathcal{V}_2\to \mathcal{V}_i$ the canonical projection, $i=1,2$. The family
        \[
            (\omega^1_1\ominus\omega^2_1,\ldots, \omega^1_k\ominus\omega^2_k)
        \]
        defined by
        $\omega^1_r\ominus\omega^2_r=\pi_1^*\omega^1_r-\pi_2^*\omega^2_r$, is a $k$-symplectic structure on $\mathcal{V}_1\times \mathcal{V}_2$.
    \end{proposition}
    \begin{proposition}\label{graph}
        Let $(\mathcal{V}_1,\omega_1^1,\ldots,\omega^1_k)$ and $(\mathcal{V}_2,\omega_1^2,\ldots, \omega^2_k)$ be two $k$-symplectic vector space and $\phi\colon \mathcal{V}_1\to \mathcal{V}_2$ a linear isomorphism. $\phi$ is a $k$-symplectomorphism if and only if its graph,
        \[
            \Gamma_\phi=\{(v_1,\phi(v_1) \, |\, v_1\in \mathcal{V}_1\}
        \]
        is a $k$-lagrangian subspace of $(\mathcal{V}_1\times \mathcal{V}_2,\omega^1_1\ominus\omega^2_1,\ldots, \omega^1_k\ominus\omega^2_k)$.
    \end{proposition}
    \begin{proof}
        We recall that
        {\small
        \begin{align*}
             \Gamma_\phi^{\bot, k}
            =&\{ (x,y)\in \mathcal{V}_1\times \mathcal{V}_2\,|\, (\omega^1_r\ominus\omega^2_r)((x,y),(v_1,\phi(v_1)))=0,\,\forall v_1\in \mathcal{V}_1,\; \forall r\}\\
            =&\{ (x,y)\in \mathcal{V}_1\times \mathcal{V}_2\,|\, \omega^1_r(x,v_1)=\omega^2_r(y,\phi(v_1)),\, \forall v_1\in \mathcal{V}_1,\, \forall r\in\{1,\ldots, k\}\}\,.
        \end{align*}}

        We now prove that $\phi$ is a symplectomorphism if and only if $\Gamma_\phi$ is a $k$-isotropic subspace.
        From Proposition \ref{charac_subspace}(iii) we know that the graph $\Gamma_\phi$ is a $k$-isotropic subspace if and only if
        \[
            (\omega^1_r\ominus\omega^2_r)((v_1,\phi(v_1)),(v_1',\phi(v_1')))=0\,,\; \forall r=1,\ldots, k\,,
        \] that is,
        {\small
        \[
            \omega^1_r(v_1,v_1')-\omega^2_r(\phi(v_1),\phi(v_1'))
            =\omega^1_r(v_1,v_1')-\phi^*\omega^2_r(v_1,v_1')=0,\; \forall r=1,\ldots, k\,,
        \]}which is equivalent to say that $\phi$ is a $k$-symplectomorphism.

        In addition, if $\Gamma_\phi$ is $k$-isotropic, it is also $k$-lagrangian. In fact, if $(x,y)\in \Gamma_\phi^{\bot, k}$, then we have
        \begin{equation}\label{aux1}
            \omega^1_r(x,v_1)=\omega^2_r(y,\phi(v_1)),\, \forall v_1\in\mathcal{V}_1,\, \forall r\in \{1,\ldots, k\}\,.
        \end{equation}
        On the other hand, $(x,\phi(x))\in \Gamma_\phi\subset\Gamma_\phi^{\bot, k}$, then
        \begin{equation}\label{aux2}
            \omega^1_r(x,v_1)=\omega^2_r(\phi(x),\phi(v_1)),\, \forall v_1\in\mathcal{V}_1,\, \forall r\in \{1,\ldots, k\}\,.
        \end{equation}

        Then from (\ref{aux1}) and (\ref{aux2}) we obtain that
        \[\omega^2(y-\phi(x),\phi(v_1))=0,\forall v_1\in \mathcal{V}_2,\;\forall r\in \{1,\ldots, k\}
        \]
        thus, as $\phi$ is an isomorphism this condition is equivalent to $y-\phi(x)\in \displaystyle\bigcap_{r=1}^k \ker\, \omega^2_r$ and as $(\mathcal{V}_2,\omega_1^2,\ldots, \omega^2_k)$ is a $k$-symplectic vector space, one obtains that $y=\phi(x)$ and therefore $\omega^1_r(x,v_1)=\phi^*\omega^2_r(x,v_1)$, i.e. $\phi$ is a $k$-symplectomorphism.
    \end{proof}
\end{subsection}

\begin{subsection}{Polarized $k$-symplectic vector space}
    It is well-know that in symplectic geometry, given any finite dimensional symplectic vector space $(V,\omega)$ and an arbitrary Lagrangian subspace $\mathcal{L}$ of $V$ (which always exists), one can construct a symplectic isomorphism between $(V,\omega)$ and $(\mathcal{L}\times \mathcal{L}^*,\omega_\mathcal{L})$, with $\omega_\mathcal{L}$ given as in  the example \ref{canonical_model}. This is not true in the general $k$-symplectic case. For instance, if we consider the example \ref{euclidean example} and the $1$-lagrangian subspace $V=span\{e_2,e_3\}$,  a trivial computation allows us to check that there is not a isomorphism between $(\mathbb{R}^3,\omega_1,\omega_2)$ and $(V\times V^*\times V^*,\omega^V_1,\omega^V_2)$, being $(\omega^V_1,\omega^V_2)$ the canonical $k$-symplectic structure given in the example \ref{canonical_model}, due to dimensional reasons.

    In this section we introduce a particular type of $k$-symplectic vector space of dimension $n(k+1)$ which are isomorphism to the canonical prototype $(V\times V^*\times\stackrel{k}{\cdots}\times V^*,\omega^V_1,\ldots, \omega^V_k)$ for some $V$ of dimension $n$.

    Consider the canonical $k$-symplectic structure
    \[
        (\mathcal{V}=V\times V^*\times\stackrel{k}{\cdots}\times V^*,\omega^V_1,\ldots,\omega^V_k)
    \]
    for some finite dimensional vector space $V$. Identifying $V$ and $V^*\times \stackrel{k}{\cdots}\times V^*$ with the subspaces $V\times \{0\}$ and $\{0\}\times V^*\times \stackrel{k}{\cdots}\times V^*$, respectively, we have the following property:
    \begin{lemma}\label{canonical lagrangian spaces}
        $V$ and $V^*\times \stackrel{k}{\cdots}\times V^*$ are complementary $k$-lagrangian subspaces of dimensions $\dim V$ and $k\dim V$, respectively.
    \end{lemma}

    Putting $W=V^*\times \stackrel{k}{\cdots}\times V^*$, we note that
    \[
        \mathcal{V}/W=V
    \]
    which, in particular, yields
    \[
        \dim V=\dim (\mathcal{V}/W) \makebox{ and } \dim W=k\dim (\mathcal{V}/W)\,.
    \]
  Next, we introduce the following definition.
    \begin{definition}
        A $k$-symplectic vector space $(\mathcal{V},\omega_1,\ldots, \omega_k)$ is said to be polarized if there exist a $k$-lagrangian subpace $W$ of $\mathcal{V}$ such that $\dim W=k\dim\big (\mathcal{V}/W\big )^*$. We denote a polarized $k$-symplectic vector space by $(\mathcal{V},\omega_1,\ldots,\omega_k,W)$.
    \end{definition}
    \begin{remark}
        Here we use the term ``polarized'' by analogy with the symplectic case, where given a real symplectic vector space $(V,\Omega)$, one can extend $\Omega$ to a complex-bilinear form $\Omega_C$ on the complexification $V_C$. A \textit{polarization} of $V$ is a Lagrangian subspace of $V_C$ and given $L$ a lagrangian subspace of $V$, its complexification $L_C$ is a lagrangian subspace of $V_C$.
    \end{remark}
    \begin{proposition}
        Let $(\mathcal{V},\omega_1,\ldots,\omega_k,W)$ be a polarized $k$-symplectic vector space. Then, there exists a $k$-lagrangian subspace $V$ which is complementary to $W$, i.e. such that $\mathcal{V}=V\oplus W$.
    \end{proposition}
    \begin{proof}
        It is a direct consequence of the definition of $k$-lagrangian subspace.
    \end{proof}
    \begin{proposition}\label{darboux lineal}
        Let $(\mathcal{V},\omega_1,\ldots,\omega_k)$ a $k$-symplectic vector space. Then $(\mathcal{V},\omega_1,\ldots,\omega_k)$ is $k$-symplectomorphic to a canonical $k$-symplectic vector space if and only if there exists a $k$-lagrangian subspace $W$ of $\mathcal{V}$ such that $\dim W=k\dim\big (\mathcal{V}/W\big )^*$.
    \end{proposition}
    \begin{proof}
        If for some vector space $V$,  there exists a $k$-symplectic isomorphism between $(\mathcal{V},\omega_1,\ldots,\omega_k)$ and $ (V\times V^*\times\stackrel{k}{\cdots}\times V^*,\omega^V_1,\ldots,\omega^V_k)$, then we consider the subspace $W$ associated to $V^*\times\stackrel{k}{\cdots}\times V^*$ via this isomorphism. It is trivial that this subspace satisfies the properties stated in the proposition.

        We now assume that the conditions of the proposition hold for some subspace $W$ of $\mathcal{V}$, i.e., we assume that $(\mathcal{V},\omega_1,\ldots,\omega_k)$ is a polarized $k$-symplectic vector space. According to the previous proposition, there exists a $k$-lagrangian subspace $V$ such that $\mathcal{V}=V\oplus W$ and $\dim V=\dim \big (\mathcal{V}/W\big)$. Then we define then the following linear mapping
        \[
            \begin{array}{rccl}
                \phi\colon & W & \to & V^*\times \stackrel{k}{\cdots}\times V^*\\\noalign{\medskip}
                & w &\mapsto & (-(\imath_w\omega_1)\vert_V,\ldots, -(\imath_w\omega_k)\vert_V)
            \end{array}
        \]

        Using the $k$-isotropic character of $W$ and that $(\omega_1,\ldots, \omega_k)$ define a $k$-symplectic structure it is easy to check that $\phi$ is injective and, since the dimensionality assumptions, we deduce that $\phi$ is in fact a linear isomorphism.

        Next, we define the mapping
        \[
            \begin{array}{rccl}
                \psi\colon & \mathcal{V}=V\oplus W & \to &V\times V^*\times \stackrel{k}{\cdots}\times V^*\\\noalign{\medskip}
                & (v,w) &\mapsto & (v,\phi(w))
            \end{array}
        \]
        which is also a linear isomorphism such that $\psi^*\omega_r^V=\omega_r$. In fact,
        \begin{align*}
          & \big(\psi^*\omega^V_r\big)\big((v,w),(\tilde{v},\tilde{w})\big) = \omega^V_r\big((v,\phi(w)),(\tilde{v},\phi(\tilde{w}))\big)\\
          =& \omega^V_r\big((v,-(\imath_w\omega_1)\vert_V,\ldots, -(\imath_w\omega_k)\vert_V), (\tilde{v},-(\imath_{\tilde{w}}\omega_1)\vert_V,\ldots, -(\imath_{\tilde{w}}\omega_k)\vert_V) \big)\\
          =& -\imath_{\tilde{w}}\omega_r(\tilde{v})+\imath_w\omega_r(v) = \omega_r\big ((v,w),(\tilde{v},\tilde{w})\big)\,.
        \end{align*}
    \end{proof}
    \begin{remark}
        A direct application of the Proposition \ref{darboux lineal} shows that given a polarized $k$-symplectic vector space $(\mathcal{V},\omega_1,\ldots, \omega_r,W)$ there exist a basis (Darboux basis) $\{e_1,\ldots, e_n,f^1_1,\ldots, f^k_1,\ldots f^1_n,\ldots, f^k_n\}$ such that $\{e_i\}$ is a basis of $V$ and $\{f^r_i\}$ is a basis of $W$. Moreover
        \[
            \omega_r=e_i^*\wedge (f^r_i)^*
        \]
        being $\{e_i^*\}$ and $\{(f^r_i)^*\}$ the dual basis of $\{e_i\}$ and $\{f^r_i\}$, respectively.
    \end{remark}

\end{subsection}
\section{$k$-symplectic manifolds}
    We turn now to the globalization of the ideas of the previous section to $k$-symplectic manifolds.

    \begin{definition}
        A $k$-symplectic manifold $(M,\omega_1,\ldots, \omega_k)$ is a family consisting of a manifold $M$ of dimension $n(k+1)$ equipped with a family of $k$ closed $2$-forms $(\omega_1,\ldots, \omega_k)$ such that  $(T_xM,\omega_1(x),\ldots, \omega_k(x))$ is a $k$-symplectic vector space for all $x\in M$. The family $(\omega_1,\ldots, \omega_k)$ is called $k$-symplectic structure.
    \end{definition}
    \begin{example}
        Let $(T^1_k)^*Q$ be the cotangent bundle of $k^1$-covelocities, i.e. the Whitney sum of $k$ copies of the cotangent bundle of a differentiable manifold $Q$, and denote by $\pi^{k,r}_Q\colon (T^1_k)^*Q\to T^*Q$ the canonical projection over the $r$-copy of the cotangent bundle. We define the family of canonical $2$-forms $(\Omega^Q_1,\ldots, \Omega^Q_k)$ as follows:
        \[
            \Omega^Q_r=(\pi^{k,r}_Q)^*\omega
        \]
        being $\omega$ the canonical symplectic form of $T^*Q$.

        A direct computation shows that $((T^1_k)^*Q,\Omega_1,\ldots, \Omega_k)$ is a $k$-symplec\-tic manifold.

        We recall that the canonical symplectic form is defined as $\omega=-d\theta_0$, where $\theta_0$ is the canonical form or Liouville form on $T^*Q$ (see for instance \cite{AM-1978}). Thus, every $\Omega^Q_r$ can be written as $\Omega^Q_r=-d\Theta^Q_r$, being $\Theta^Q_r=(\pi^{k,r}_Q)^*\theta_0$.
    \end{example}
    \begin{remark}\label{relation_canonical_models}
        For each $\alpha_x\in (T^1_k)^*Q= T^*Q\oplus\stackrel{k}{\cdots}\oplus T^*Q$, the $k$-symplectic vector space $(T_{\alpha_x}(T^*Q\oplus\stackrel{k}{\cdots}\oplus T^*Q),\Omega^Q_1(\alpha_x),\ldots, \Omega^Q_k(\alpha_x))$ associated to the $k$-symplectic manifold $((T^1_k)^*Q,\Omega^Q_1,\ldots, \Omega^Q_k)$ is related with the canonical $k$-symplectic structure on $T_xQ\times T_x^*Q\times\stackrel{k}{\cdots}\times T_x^*Q$ described in example \ref{canonical_model} with $V=T_xQ$. In fact, for each $\alpha_x\in T^*Q\oplus\stackrel{k}{\cdots} T^*Q$, we consider the following vector space isomorphism 
        \[
            \begin{array}{rcl}
                T_{\alpha_x}(T^*Q\oplus\stackrel{k}{\cdots}\oplus T^*Q) & \stackrel{\Psi_{\alpha_x}}{\longrightarrow} & T_xQ\times T_x^*Q\times\stackrel{k}{\ldots}\times T_x^*Q\\\noalign{\medskip}
                Z_{\alpha_x} & \mapsto & ((\pi^k_Q)_*(\alpha_x)(Z_{\alpha_x}), \beta^1_x,\ldots,\beta^k_x)
            \end{array}
        \]
        where $\beta^r_x\in T^*_xQ$ is the covector such that
        \[
             (\beta^r_x)^{(r)}=(\pi^{k,r}_Q)_*(\alpha_x)(Z_{\alpha_x})\in T_{\alpha^r_x}(T^*Q)
        \] being $(\beta^r_x)^{(s)}$ the $s$-vertical lift of $\beta^r_x$  to $T^*Q\oplus\stackrel{k}{\cdots}\oplus T^*Q$ defined in \cite{LMS-1988}. We now recall this definition:  
            \begin{quote}
                Consider the following commutative diagram:
                \[
                    \xymatrix{
                        T^*Q\oplus\stackrel{k}{\cdots}\oplus T^*Q\ar[d]_-{\pi^k_Q}\ar[r]^-{\pi^{k,r}_Q} & T^*Q \ar[dl]^-{\pi_Q}\\
                        Q
                    }
                \]

                Suppose that $\alpha_x\in T^*Q\oplus\stackrel{k}{\cdots}\oplus T^*Q$, $\pi^k_Q(\alpha_x)=x$, $ \pi^{k,r}_Q(\alpha_x)=\alpha_x^r$ and that $\beta\in T^*_xQ$. The $(s)$-vertical lift of $\beta$ to $T^*Q\oplus\stackrel{k}{\cdots}\oplus T^*Q$ is the unique tangent vector $\beta^{(s)}\in V_r$, such that
                \[
                    \imath_{(\pi^{k,s}_Q)_*(\alpha_x)(\beta^{(s)})}\omega=(\pi_Q)^*\beta
                \]
                being $\omega$ the canonical symplectic form on $T^*Q$.

                Locally, we consider a local coordinate system $(q^i,p^r_i)$ on $T^*Q\oplus\stackrel{k}{\cdots}\oplus T^*Q$. If $\beta=\beta_idq^i$, then we have
                \[
                    \beta^{(s)}=\beta_i\displaystyle\frac{\partial}{\partial p^s_i}.
                \]
            \end{quote}

        Moreover, if $(\Omega^Q_1(\alpha_x),\ldots, \Omega^Q_k(\alpha_x))$ is the canonical $k$-symplectic structure on the vector space $T_{\alpha_x}(T^*Q\oplus\stackrel{k}{\ldots}\oplus T^*Q)$ and $(\omega^{T_xQ}_1,\ldots, \omega^{T_xQ}_k)$ the canonical $k$-symplectic structure on  $T_xQ\times T^*_xQ\times \stackrel{k}{\ldots}\times T^*_xQ$ introduced in the example \ref{canonical_model}, then
        \begin{equation}\label{relation}
            \left(\Psi_{\alpha_x}\right)^*\left(\omega^{T_xQ}_r\right) = \Omega^Q_r(\alpha_x),\; \forall r\in \{1,\ldots, r\}.
        \end{equation}
    \end{remark}

\subsection{Isotropic, coisotropic and Lagrangian submanifolds}

    Following the notion of special submanifolds in the symplectic case we can give the following definition:
    \begin{definition}\label{special}
        Let $N$ be a submanifold of a $k$-symplectic manifold $(M,\omega_1,\ldots, \omega_k)$. $N$ is said to be $l$-isotropic (resp. $l$-coisotropic, $l$-lagrangian)  if $T_xN$ is a $l$-isotropic (resp. $l$-coisotropic, $l$-lagrangian) vector subspace of the $k$-symplectic vector space $(T_xM,\omega_1(x),\ldots, \omega_k(x))$ for all $x\in N$.
    \end{definition}

    \begin{proposition}\label{fibers}\
        \begin{enumerate}
            \item The fibers of $\pi^k_Q\colon T^*Q\oplus\stackrel{k}{\cdots}\oplus T^*Q\to Q$ are $k$-lagrangian.
            \item The image of a section $\gamma$ of $\pi^k_Q$ is $k$-lagrangian if and only if $\gamma$ is a closed section.
        \end{enumerate}
    \end{proposition}
    \begin{proof}
        (i)\; In the above remark we recall that at each point $\alpha_x\in (T^1_k)^*Q$ we have
        \[
            T_{\alpha_x}(T^*Q\oplus\stackrel{k}{\cdots}\oplus T^*Q) \cong T_xQ\times T_x^*Q\times\stackrel{k}{\ldots}\times T_x^*Q\,.
        \]

        On the other hand, for the tangent space to the fibre through $\alpha_x$ we deduce
        \[
            T_{\alpha_x}\big ( (\pi^k_Q)^{-1}(\pi^k_Q(\alpha_x))\big)\cong T_x^*Q\times\stackrel{k}{\ldots}\times T_x^*Q
        \]
         and, hence, it follows from Lemma \ref{canonical lagrangian spaces} that this is a $k$-lagrangian subspace of $T_{\alpha_x}(T^*Q\oplus\stackrel{k}{\cdots}\oplus T^*Q)$.

         (ii)\; Now we prove that $L$ is a $k$-isotropic  submanifold if and only if $\gamma$ is a closed section of $\pi^k_Q$.
        Let us observe that if $\gamma$ is a section of $\pi^k_Q$, then there exist a family of $k$ $1$-forms on $Q$ such that $\gamma =(\gamma_1,\ldots, \gamma_k)$. A section $\gamma$ satisfies
         \begin{equation}\label{aux3}
            \gamma^*\Theta^Q_r=\gamma^*\left( (\pi^{k,r}_Q)^*\theta_0\right) = \gamma_r^*\theta_0=\gamma_r,
          \end{equation}
          where in the last identity we use the following universal property of $\theta_0$:
            \begin{quote}
                \textit{The canonical form $\theta_0$ on $T^*Q$ is the unique one-form with the property that, for any one-form $\beta$ on $Q$, $\beta^*\theta_0=\beta$.}
            \end{quote}
          From (\ref{aux3}) one obtains that $d\gamma_r=\gamma^*d\Theta^Q_r=-\gamma^*\Omega^Q_r$. Thus $\gamma$ is a closed section, i.e., each $\gamma_r$ is a closed $1$-form, if and only if $\gamma^*\Omega^Q_r=0$, for all $r \in \{1,\ldots, k\}$, or equivalently $\Omega^Q_r\vert_{\gamma(Q) \times \gamma(Q)}=0$ for all $r\in \{1,\ldots, k\}$. Finally, by the Proposition \ref{charac_subspace}(iii) we know that it is equivalent to say that $\gamma(Q)$ is $k$-isotropic, but in this particular case there exist a $k$-isotropic complement given by the fibers of the projection, thus, $L=\gamma(Q)$ is a $k$-lagrangian submanifold.
    \end{proof}

    In particular, note that the zero section of $\pi^k_Q$ is a $k$-lagrangian submanifold.
    \begin{remark} 
        In the previous proposition, since $\gamma$ is closed, i.e., each $\gamma_r$ is closed, for each $r$, we have that every point has an open neighborhood $U\subset Q$ where there exist $k$ functions $W_r\in \mathcal{C}^{\infty}(U)$ such that $\gamma_r=dW_r$. This functions $W_1,\ldots, W_k$ are called the \textit{characteristic functions}. The idea of characteristic functions go back to Hamilton-Jacobi theory. In the $k$-symplectic framework, given a Hamiltonian function $H\in\mathcal{C}^{\infty}(T^*Q\oplus\stackrel{k}{\cdots}\oplus T^*Q)$, the Hamilton-Jacobi problem consists of finding $k$ function $W_1,\ldots, W_k\colon U\subset Q\to \R$ such  that
            \[
                H\left(q^i,\displaystyle\frac{\partial W_1}{\partial q^i},\ldots, \displaystyle\frac{\partial W_k}{\partial q^i}\right)=\makebox{\rm constant}.
            \]

        In \cite{LMMSV-2010}, we give the geometric version of this equation as follow: let $\gamma=(\gamma_1,\ldots, \gamma_k)$ a closed section of $\pi^k_Q\colon T^*Q\oplus\stackrel{k}{\cdots}\oplus T^*Q\to Q$  with $\gamma_r=dW_r$. If $d(H\circ \gamma)=0$, then $W_1,\ldots, W_k$ is a solution of the Hamilton-Jacobi problem in the $k$-symplectic approach.
    \end{remark}

        As in \ref{canonical_relation} and \ref{graph} one has
    \begin{proposition}\label{canonical_relation2}
        Let $(M_1,\omega^1_1,\ldots,\omega^1_k)$ and $(M_2,\omega^2_1,\ldots,\omega^2_k)$ be two $k$-symplectic manifolds, $\pi_i\colon M_1\times M_2\to M_i$ the canonical projection onto $M_i,\, i=1,2$ and for each $r\in \{1,\ldots, k\}$
        \[
            \omega^1_r\ominus\omega^2_r=\pi_1^*\omega^1_r-\pi_2^*\omega^2_r\,.
        \]
        Then:
        \begin{enumerate}
            \item $(\omega^1_1\ominus\omega^2_1,\ldots, \omega^1_k\ominus\omega^2_k)$ is a $k$-symplectic structure on $M_1\times M_2$.
            \item A map $\phi\colon M_1\to M_2$ is a $k$-symplectomorphism if and only if its graph $\Gamma_\phi$ is a $k$-lagrangian submanifold.
        \end{enumerate}
    \end{proposition}
    \begin{proof}
        It is easy to prove (i). To prove (ii), note that $\phi$ induce a diffeomorphism of $M_1$ to
        \[
            \Gamma_\phi=\{ (x,\phi(x))\, | \, x\in M_1\},
        \]
        so we can write
        \[
            T_{(x,\phi(x))}\Gamma_\phi=\left\{
            (v_x,\phi_*(x)(v_x)) \, | \, v_x\in T_xM_1
            \right\} = graph( \phi_*(x))\,.
        \]

 \noindent
 $\phi$ is a $k$-symplectomorphism if and only if $\phi^*\omega^2_r=\omega^1_r$ for all $r\in \{1,\ldots, k\}$, i.e., for each $x\in M_1$, the linear isomorphism \[
         \phi_*(x)\colon T_xM_1\to T_{\phi(x)}M_2\] is a $k$-symplectomorphism; now by Proposition \ref{graph} that this is equivalent to the fact that $graph (\phi_*(x))$ is a $k$-lagrangian vector subspace of $(T_xM_1\times T_{\phi(x)}M_2,\omega^1_1(x)\ominus\omega^2_1(\phi(x)),\ldots, \omega^1_k(x)\ominus\omega^2_k(\phi(x)))$. Finally, by definition \ref{subspaces}, we know that it is equivalent to say that $\Gamma_\phi$ is a $k$-lagrangian submanifold of $M_1\times M_2$.
    \end{proof}

    Proposition \ref{fibers} tell us that the $k$-symplectic manifold $((T^1_k)^*Q,\omega_1,\ldots,$ $ \omega_k)$ possesses a $k$-lagrangian foliation\footnote{ A $k$-lagrangian foliation is a foliation such that its  leaves are $k$-lagrangian submanifolds.} with a transversal $k$-lagrangian section. Moreover, the leaves of this foliation all have the same dimension, namely
    \[
        \dim (\pi^{k}_Q)^{-1}(x) = \dim T^*_xQ\times \stackrel{k}{\cdots}\times T^*_xQ\,.
    \]

    These observations prompt us to introduce the following definition.
    \begin{definition}
        A family $(M,\omega_1,\ldots, \omega_k,W)$ consisting of a $k$-symplec\-tic manifold $(M,\omega_1,\ldots, \omega_k)$ of dimension $n(k+1)$ and $W$ is a $k$-lagrangian involutive $n$-codimensional distribution on  $(M,\omega_1,\ldots, \omega_k)$, i.e.  for each $x\in M$, $W(x)$ is a $k$-lagrangian subspace of $T_xM$, is called a polarized $k$-symplectic manifold. We denote by $\mathcal{F}$ the foliation defined by the subbundle $W$.
    \end{definition}

    A big difference between symplectic and $k$-symplectic manifolds is that in the $k$-symplectic case a theorem type Darboux is only valid for the polarized $k$-symplectic manifolds. This theorem has been proved in \cite{Aw-1992,{LMS-1993}}. We now recall this theorem
    \begin{theorem}\textbf{(Darboux theorem)}
        Let $(M,\omega_1,\ldots, \omega_k,W)$ be a polarized $k$-symplectic manifold. About every point of $M$ we can find a local coordinate system $(x^i,y^r_i),\, 1\leq i\leq n,\, 1\leq r\leq k$, called adapted coordinate system, such that
        \[
            \omega_r=\ds\sum_{i=1}^ndx^i\wedge dy^r_i
        \]
        for each $1\leq r\leq k$ and
        \[
            W(x)=span\big\{ \ds\frac{\partial}{\partial y^r_i},\, 1\leq i\leq n,\, 1\leq r\leq k\big \}\,.
        \]
    \end{theorem}
    \begin{remark}
        In this section we introduce the notion of polarized $k$-symplectic manifolds, this structure is called $k$-symplectic structure by Awane \cite{Aw-1992,Aw-1998} and it is equivalent to the notion of standard polysymplectic structure of Gunther \cite{Gu-1987} and integrable $p$-almost cotangent structure introduced by M. de Le\'{o}n \textit{et al.} \cite{LMS-1988, LMS-1993}
    \end{remark}

\subsection{Normal form for $k$-lagrangian submanifolds}\label{normal_form_section}

    There is an important theorem due to A. Weinstein which gives the normal form for a lagrangian submanifold $\mathcal{L}$ in a symplectic manifold $(\mathcal{P},\omega)$.
        \begin{theorem} (\textbf{A. Weinstein \cite{W_1971}})
            Let $(\mathcal{P},\omega)$ be a symplectic manifold and let $\mathcal{L}$ be a lagrangian submanifold. Then there exists a tubular neighborhood $U$ of $\mathcal{L}$ in $\mathcal{P}$, and a diffeomorphism $\phi\colon U\to V=\phi(U)\subset T^*\mathcal{L}$ into an open neighborhood $V$ of the zero cross-section in $T^*\mathcal{L}$ such that $\phi^*(\omega_\mathcal{L}\vert_{V})=\omega\vert_{U}$, where $\omega_\mathcal{L}$ is the canonical symplectic form on $T^*\mathcal{L}$.
        \end{theorem}
    Now we will extend to the $k$-symplectic setting this important theorem due to A. Weinstein. Before, we recall the relative Poincar\'e lemma \cite{W_1971}, which will be useful in the sequel.
    \begin{lemma}\label{Poincare} (\textbf{Relative Poincar\'e lemma})
        Let $N$ be a submanifold of a differentiable manifold $M$, and let $U$ be a tubular neighborhood  of $N$ with bundle map $\pi_0\colon U\to N$. Notice that $\pi_0$ is a vector bundle. Denote by $\Delta$ the dilation vector field of this vector bundle, and let $\psi_s$ be the multiplication by $s$. If we define an integral operator on forms on $U$ as follows
        \[
            I(\Omega)_p=\int_0^1\imath_{\Delta_s}\psi_s^*\Omega_pds
        \]
        where $\Delta_s=\frac{1}{s}\Delta$ and $p\in U$, then we have
        \[
            I(d\Omega)+d(I\Omega)=\Omega-\pi_0^*\left(\Omega\vert_N\right)
        \]
        being $\Omega\vert_N$ the form on $N$ obtained by restricting $\Omega$ pointwise to $TN$ (observe that $U$ can be taken as a normal bundle of $TN$ in $M$).
    \end{lemma}

    \begin{theorem}\label{normal_form}
        Let $(M,\omega_1,\ldots, \omega_k, W)$ be a polarized $k$-symplectic manifold. Let $\mathcal{L}$ be a $k$-lagrangian submanifold which is complementary to $W$, that is,  $T\mathcal{L}\oplus W\vert_\mathcal{L}=TM\vert_\mathcal{L}$. Then there is a tubular neighborhood $U$ of $\mathcal{L}$ and a diffeomorphism $\Phi\colon U\to V\subset T^*\mathcal{L}\oplus\stackrel{k}{\cdots}\oplus T^*\mathcal{L}$ where $V$ is a neighborhood of the zero section, such that $\Phi\vert_\mathcal{L}$ is the standard identification of $\mathcal{L}$ with the zero section of $T^*\mathcal{L}\oplus\stackrel{k}{\cdots}\oplus T^*\mathcal{L}$, and
        \[
            \Phi^*\left(\Omega^\mathcal{L}_r\vert_V\right)=
            \omega_r\vert_U,
        \]for all $r\in\{1,\ldots, k\}$, where $(\Omega^\mathcal{L}_1,\ldots, \Omega^\mathcal{L}_k)$ is the canonical $k$-symplectic structure on $T^*\mathcal{L}\oplus\stackrel{k}{\cdots}\oplus T^*\mathcal{L}$.
    \end{theorem}
    \begin{proof}
         $\mathcal{L}$ is a $k$-lagrangian submanifold of $M$ complementary to $W$, that is, for each $x\in \mathcal{L}$, $T_x\mathcal{L}$ is a $k$-lagrangian subspaces of $T_xM$ such that $T_x\mathcal{L}\subset (T_x\mathcal{L})^{\bot, k}$ and $T_xM=T_x\mathcal{L}\oplus W(x)$, being $W(x)$ a $k$-lagrangian subspace of $T_xM$

        Firstly, we define a vector bundle morphism over the identity of $\mathcal{L}$
        \[
             \xymatrix{ W \ar[rr]^-{\phi}\ar[rd]_-{\pi_0} &  &T^*\mathcal{L}\oplus \stackrel{k}{\cdots}\oplus T^*\mathcal{L}\ar[ld]^-{\pi^k_\mathcal{L}}\\ &\mathcal{L}&}
        \]
        where
        \[
            \phi(w)=(-\imath_w\omega_1,\ldots,-\imath_w\omega_k)\,.
        \]

        The morphism $\phi$ is injective. In fact,
        \[
            \ker\,\phi=\bigcap_{r=1}^k\ker\,\omega_r=\{0\}
        \]
        where in the last identity we use that $(\omega_1,\ldots, \omega_k)$ is a $k$-symplectic structure on $M$. Now by dimensionally assumptions, we deduce that $\phi$ is a vector bundle isomorphism.

        Since $TM\vert_\mathcal{L}=T\mathcal{L}\oplus W\vert_{\mathcal{L}}$,  $\phi$ induces a morphism on a tubular neighborhood $U$ defined by $W$ onto a neighborhood of $\mathcal{L}$ in $T^*\mathcal{L}\oplus\stackrel{k}{\cdots}\oplus T^*\mathcal{L}$ (as usual, the latter embedding is understood as the identification of $\mathcal{L}$ with the zero section). We shall denote the restriction of $\phi$ to the tubular neighborhood $U$ by $\varphi$. Notice that the restriction of $\varphi$ to $\mathcal{L}$ is the identity, so that $T\varphi$ is also the identity on $T\mathcal{L}$. Since $\phi$ is injective, then $\varphi\colon U\to \varphi(U)$ is a bundle isomorphism.

        Using the identifications $TM\vert_{\mathcal{L}}=T\mathcal{L}\oplus W\vert_{\mathcal{L}}$ and $T_{\alpha_x}(T^*\mathcal{L}\oplus\stackrel{k}{\cdots}\oplus T^*\mathcal{L})=T_x\mathcal{L}\times T^*_x\mathcal{L}\times\stackrel{k}{\cdots}\times T^*_x\mathcal{L}$\footnote{The description of this identification was given in remark \ref{relation_canonical_models}} and the composition $\Xi= \Psi^{-1}\circ (Id_{T_x\mathcal{L}},\varphi)$ given by
        \[
            \Xi(v_x+w_x)=\Psi^{-1}(v_x,-\imath_{w_x}\omega_1,\ldots, -\imath_{w_x}\omega_k),
        \]
        we have
        \begin{align*}
             & \Xi^*\left(\Omega^\mathcal{L}_r(\alpha_x)\right) (v_x+w_x,v'_x+w'_x)
             \\ =&
             (Id_{T_x\mathcal{L}},\varphi)^*\left((\Psi^{(-1)})^*\left(\Omega^\mathcal{L}_r(\alpha_x)\right) \right) (v_x+w_x,v'_x+w'_x)\\
             =&
             (Id_{T_x\mathcal{L}},\varphi)^*(\omega^{T_x\mathcal{L}}_r) (v_x+w_x,v'_x+w'_x)
             \\ =&
             \omega^{T_x\mathcal{L}}_r\left( (v_x,-\imath_{w_x}\omega_1,\ldots,-\imath_{w_x}\omega_k), (v'_x,-\imath_{w'_x}\omega_1,\ldots,-\imath_{w'_x}\omega_k)\right)
             \\
            =&
                \left(-\imath_{w'_x}\omega_r\right)(v_x)- \left(-\imath_{w_x}\omega_r\right)(v'_x) = \omega_r(x)(v_x,w'_x) + \omega_r(x)(w_x,v'_x)\\
            =&
                \omega_r(x)(v_x+w_x,v'_x+w'_x)\,,
        \end{align*}
    where  we have used that $(\Psi^{-1})^*(\Omega^L_r(\alpha_x))= \omega^{T_xL}_r$ (see (\ref{relation}) and in the last identity that $T_xL$ and $W_x$ are two $k$-isotropic subspaces. Then we have
        \[
            \Xi^*\Omega^L_r=\omega_r \makebox{ on } \mathcal{L}.
        \]

    Now we use $\varphi\colon U\to \varphi(U)$ to pushforward $\omega_1,\ldots, \omega_r$ and  we obtain a family of $k$ $2$-forms $\Omega_1,\ldots, \Omega_k$ in a neighborhood of $\mathcal{L}$ in $T^*\mathcal{L}\oplus\stackrel{k}{\ldots}\oplus T^*\mathcal{L}$. Using Lemma \ref{Poincare} we deduce that each $\Omega_r=d\Theta_r$, where $\Theta_r=I(\Omega_r)$, and from the definition of $I$ one obtains
    \begin{equation}\label{rel_theta}
        \Theta^\mathcal{L}_r\vert_{\mathcal{L}}=\Theta_r\vert_\mathcal{L} =0.
    \end{equation}

    Define the family of $k$ $2$-forms defined in a neighborhood of $\mathcal{L}$ in $T^*\mathcal{L}\oplus\stackrel{k}{\ldots}\oplus T^*\mathcal{L}$.
    \[
        \Omega_{r,s}=\Omega^\mathcal{L}_r + s(\Omega_r-\Omega^\mathcal{L}_r),\; s\in [0,1].
    \]

    It is easy to check that
    \[
        \Omega_{r,s}\vert_\mathcal{L}=
        \Omega^\mathcal{L}_r\vert_\mathcal{L}=\Omega_r\vert_\mathcal{L}\,.
    \]
    Then if $x\in \mathcal{L}$ then $\displaystyle \bigcap_{r=1}^k\ker \Omega_{r,s}(x)=\displaystyle \bigcap_{r=1}^k\ker \Omega^\mathcal{L}_r(x)=\{0\}$ because $(\Omega^\mathcal{L}_1,\ldots, \Omega^\mathcal{L}_k)$ is a $k$-symplectic structure.  We can find a neighborhood of $\mathcal{L}$ on $T^*\mathcal{L}\oplus\stackrel{k}{\cdots}\oplus T^*\mathcal{L}$ on which $\displaystyle \bigcap_{r=1}^k\ker \Omega_{r,s}=0$ for all $s\in [0, 1]$, that is, in this neighborhood $(\Omega_{1,s},\ldots, \Omega_{k,s})$ is a $k$-symplectic structure, for all $s$.

    From the property (\ref{rel_theta}) we deduce that there is a unique time-depending $\pi^k_\mathcal{L}$-vertical vector field $X_s$ such that
    \[
        \imath_{X_s}\Omega_{r,s}=\Theta_r-\Theta^\mathcal{L}_r
    \]

    Moreover, the vector field $X_s$ vanishes on $\mathcal{L}$, thus we can neighborhood of $\mathcal{L}$ in $T^*\mathcal{L}\oplus \stackrel{k}{\ldots}\oplus T^*\mathcal{L}$ such that the flow $\psi_s$ of $X_s$ is defined at least for $s\leq 1$. Therefore
    \begin{align*}
        \displaystyle\frac{d}{ds}\left(\psi_s^*\Omega_{r,s}\right) &= \psi_s^*\left(\mathcal{L}_{X_s}\Omega_{r,s}\right) + \psi_s^*\left(\frac{d\Omega_{r,s}}{ds}\right)   \\
        &= \psi_s^*\left( d\imath_{X_s}\Omega_{r,s}\right) +  \psi_s^*\left( \Omega_r-\Omega^\mathcal{L}_r\right) \\
        &= \psi_s^*\left( d\Theta_r+d\Theta^\mathcal{L}_R+\Omega_r-\Omega^\mathcal{L}_r\right)=0\,.
    \end{align*}
    Then,
    \[
        \psi_1^*\Omega_{r,1}=\psi_0^*\Omega^\mathcal{L}_r=\Omega^\mathcal{L}_r.
    \]

    Finally, as $X_s\vert_\mathcal{L}=0$, $(\psi_s)\vert_\mathcal{L}=id_\mathcal{L}$ and then we deduce that $\psi_1\circ \Xi$ gives us the desired local diffeormorphism.
    \end{proof}

    This result is similar to the theorem which gives the normal form of a lagrangian submanifold in the multisymplectic setting, see for instance \cite{LMS-2007, Martin}.

    As a consequence of the above theorem we obtain a \textit{Equivalence Theorem for $k$-lagrangian submanifolds}.
    \begin{theorem}\label{equivalence}
        Let $(M_i,\omega^i_1,\ldots, \omega^i_k, W_i), \, (i=1,2)$ be two polarized $k$-symplectic manifolds such that $\mathcal{L}$ is a $k$-lagrangian submanifold of each $M_i$ complementary to $W_i$. Then there exist a diffeomorphism $\Phi\colon U\subset M_1\to V=\Phi(U)\subset M_2$ from a neighborhood $U$ of $\mathcal{L}$ in $M_1$ in a neigborhood $V$ of $\mathcal{L}$ in $M_2$, such that $\Phi\vert_\mathcal{L}=Id_\mathcal{L}$ and $\Phi^*\left(\omega^2_r\vert_V\right) =\omega^1_r\vert_U$ for each $r\in \{1,\ldots, k\}$.
    \end{theorem}
    \begin{proof}
        From Theorem \ref{normal_form} we know that for each $i=1,2$ there exist a tubular neighborhood $U_i$ of $\mathcal{L}$ and a diffeomorphism $\Phi_i\colon U_i\to V_i=\Phi_i(U_i)\subset T^*\mathcal{L}\oplus\stackrel{k}{\cdots}\oplus T^*\mathcal{L}$ such that $\Phi_i\vert_\mathcal{L}=Id_\mathcal{L}$ (here we consider the standard identification of $\mathcal{L}$ with the zero section of $T^*\mathcal{L}\oplus\stackrel{k}{\cdots}\oplus T^*\mathcal{L}$) and
        \[
            (\Phi_i)^*\left(\Omega^{\mathcal{L}}_r\vert_{V_i}\right)=\omega^i_r\vert_{U_i}\,.
        \]

        We now consider the composition
        \[
            \xymatrix{ U=\Phi_1^{-1}(V_1\cap V_2)\ar[r]^-{\Phi_1}\ar@/^{10mm}/[rr]^-{\Phi} & V_1\cap V_2\ar[r]^-{\Phi_2^{-1}} & V=\Phi_2^{-1}(V_1\cap V_2)
            }
        \]

        Then $\Phi\vert_\mathcal{L}=Id_\mathcal{L}$ and
        \[
            \Phi^*\left(\omega^2_r\vert_{V}\right)=\Phi_1^*\left( (\Phi_2^{-1})^*(\omega^2_r\vert_{U}) \right) = \Phi_1^*\left( \Omega^{\mathcal{L}}_r\vert_{V_1\cap V_2}\right) = \omega^1_r\vert_{U}\,.
        \]
    \end{proof}

\section{The group of diffeomorphism of a $k$-symplectic manifolds}

    The diffeomorphisms of a manifold $M$ into itseft may be identified with their graphs, i.e., with the submanifolds of $M\times M$ which are mapped diffeomorphically onto $M$ by both of the projections $\pi_1$ and $\pi_2$.

    If $M$ has a $k$-symplectic structure $(\omega_1,\ldots, \omega_k)$, $M\times M$ has the $k$-symplectic structure $(\omega_1\ominus\omega_1,\ldots, \omega_k\ominus\omega_k)$ introduced in Proposition \ref{canonical_relation2}. In that Proposition we prove that a diffeomorphism $\phi\colon M\to M$ is a $k$-symplectomorphism if and only if its graph $\Gamma_\phi$ is a $k$-lagrangian submanifold of $M\times M$. Therefore, if $\Delta$ denotes the graph of the identity $Id_M$, then $\Delta$ is a $k$-lagrangian submanifold of $M\times M$, and, by Theorem \ref{normal_form}, there is a tubular neighborhood $U$ of $\Delta$ in $M\times M$ and a diffeomorphism $\Phi\colon U\to V=\Phi(U)\subset T^*\Delta\oplus\stackrel{k}{\cdots}\oplus T^*\Delta$ such that $\Phi\vert_\Delta$ is the standard identification of $\Delta$ with the zero section of $T^*\Delta\oplus\stackrel{k}{\cdots}\oplus T^*\Delta$ and
    \[
        \Phi^*\left( \Omega^\Delta_r\vert_V\right) = (\omega_r\ominus\omega_r)\vert_U\,,
    \]
    for all $r\in \{1,\ldots, k\}$.

    On the other hand, $Id_M$ induces the following diffeomorphism of $M$ to $\Gamma_{Id_M}$
    \[
        \begin{array}{rcl}
            f\colon M & \to & \Delta\\
            x & \mapsto & (x,x)
        \end{array}
    \]

    We now consider the canonical prolongation $(T^1_k)^*f$ of this diffeomorphism to the bundle of $k^1$-covelocities from $T^*\Delta\oplus\stackrel{k}{\cdots}\oplus T^*\Delta$ to $T^*M\oplus\stackrel{k}{\cdots}\oplus T^*M$. Thus, the composititon of $\Phi$ and $(T^1_k)^*f$
    {\footnotesize
    \[
        \xymatrix{ U\ar[r]^-{\Phi} & V=\Phi(U)\subset T^*\Delta\oplus\stackrel{k}{\cdots}\oplus T^*\Delta \ar[r]^-{(T^1_k)^*f} & W=(T^1_k)^*f(V)\subset T^*M\oplus\stackrel{k}{\cdots}\oplus T^*M
        }
    \] }
    gives us a $k$-symplectic diffeomorphism from a tubular neigborhood $U$ of $\Delta$ in $M\times M$ to a neighborhood $W$ of the zero section of $T^*M\oplus\stackrel{k}{\cdots}\oplus T^*M$.

    Thus, the diffeomorphisms ``near'' the identity $Id_M$ are thereby put in one to one correspondence with a neighborhood of the zero in the space of sections of $\pi^k_M\colon T^*M\oplus\stackrel{k}{\cdots}\oplus T^*M \to M$ in such  a way that the $k$-symplectic diffeomorphism go onto the subspace of closed sections of $\pi^k_M$, that is, in the subspace of families of $k$ closed  $1$-forms over the same basis point. This gives a coordinate chart for the diffeomorphism group around the identity of $M$ which the $k$-symplectic automorphism group goes onto a linear subspace. Thus the $k$-symplectic automorphism group of $M$  is a manifold modeled on the space of closed sections of $\pi^k_M$. In this discussion we consider the appropriate topologies and differentiable structures on the spaces of diffeomorphism and closed sections. In a future work we want describe with detail the group of automorphism of a $k$-symplectic manifold.
\section{Conclusions and outlook}

   In this paper we have discussed some relevant properties of the $k$-symplectic geometry. First, we have studied some properties of $k$-symplectic vector spaces and their linear subspaces.
Indeed, we have introduced the different kinds of orthogonal complement of a subspaces and extended many of the results in symplectic linear spaces to our case.
These results are later extended for $k$-symplectic manifolds.

   The main result of this paper is the generalization of the Weinstein's normal form theorem for lagrangian submanifolds to this new geometric setting. A direct consequence is a local equivalence theorem for $k$-lagrangian submanifolds. A second consequence is that we can induce a differentiable structure in the automorphism group of a $k$-symplectic manifolds. The considered $k$-symplectic manifolds need an extra geometric ingredient (a polarization); indeed, this theorem is only valid for polarized $k$-symplectic manifolds, that is a $k$-symplectic manifold $M$ of dimension $n(k+1)$ with a $k$-lagrangian integrable $n$-codimensional distribution $W$.

   As is well-known, some classical geometrical structures are determined by their automorphism groups, for instance it was shown by Banyaga \cite{Banyaga-1978, Banyaga-1986, Banyaga-1988} that the geometric stuctures defined by a volume or a symplectic form on a differentiable manifold are determined by their automorphism groups, the groups of volume preserving and symplectic diffeomorphism respectively. The natural question is if we can obtain a similar result in our context, that is, if the geometric structures defined by a $k$-structure are determined by the automorphism group of a $k$-symplectic manifold. We hope to use the normal form theorem for $k$-lagrangian submanifolds of a $k$-symplectic manifold to go further in the study of the group of the automorphism of a $k$-symplectic manifolds. In order to simplify this paper, we leave this study for a future paper.


\end{document}